\documentclass[12pt,english]{amsart}
\usepackage[a4paper]{geometry}
\usepackage{amssymb,amsmath}
\usepackage{enumerate}
\makeatletter

\usepackage{latexsym}
\usepackage{color}

\newtheorem{thm}{Theorem}[section]
\newtheorem{lemme}[thm]{Lemma}

\newtheorem{rem}[thm]{Remark}

\def\B{{\mathbb{B}}}
\def\S{{\mathbb{S}}}

\def\R{{\mathbb{R}}}
\def\C{{\mathbb{C}}}

\def\N{{\mathbb{N}}}

\numberwithin{equation}{section} 
\numberwithin{figure}{section} 
\numberwithin{table}{section} 

\usepackage[bookmarksopen, naturalnames]{hyperref}

\makeatother

\usepackage{babel}

\begin{document}
	\title{Construction of labyrinths in pseudoconvex domains}
	\author{S. Charpentier, \L . Kosi\'nski}
	\thanks{The first author is partly supported by the grant ANR-17-CE40-0021
		of the French National Research Agency ANR (project Front). The second author is partly supported by the NCN grant SONATA BIS no. 2017/26/E/ST1/00723}
	\address{St\'ephane Charpentier, Institut de Mathematiques, UMR 7373, Aix-Marseille
		Universite, 39 rue F. Joliot Curie, 13453 Marseille Cedex 13, France}
	\email{stephane.charpentier.1@univ-amu.fr}
	\address{\L ukasz Kosi\'nski, Institute of Mathematics, Jagiellonian University, \L ojasiewicza 6, 30-348 Krak\'ow, Poland}
	\email{lukasz.kosinski@im.uj.edu.pl}
	
	\begin{abstract}We build in a given pseudoconvex (Runge) domain $D$ of $\C^N$ a $\mathcal O(D)$ convex set $\Gamma$, every connected component of which is a holomorphically contractible (convex) compact set, enjoying the property that any continuous path $\gamma:[0,1)\rightarrow D$ with $\lim _{r\rightarrow 1}\gamma(r)\in \partial D$ and omitting $\Gamma$ has infinite length. This solves a problem left open in a recent paper by Alarc\'on and Forstneri\v{c}.
	\end{abstract}
	
	\maketitle

\section{Introduction}
Alarc\'on and Forstneri\v{c} recently proved that the Euclidean ball $\B_N$ of $\C^N$, $N>1$, admits a nonsingular
holomorphic foliation by complete properly embedded holomorphic discs \cite[Theorem 1]{AlaFor}. They asked the natural question whether their result extends to any Runge pseudoconvex domains.
As explained in \cite[Remark 1]{AlaFor}, the main obstruction that appears is how to construct a \emph{suitable labyrinth} in such a domain. Here and in the sequel we call labyrinth of a given pseudoconvex domain $D$ in $\C^N$ a set $\Gamma$ in $D$ with the property that any continuous path $\gamma:[0,1)\rightarrow D$, with $\lim _{r\rightarrow \infty}\gamma(r) \in \partial D$, whose image does not intersect $\Gamma$, has infinite length. Such sets were already built in pseudoconvex domains by Globevnik, by properly embedding the pseudoconvex domain as a submanifold of $\C^{2N+1}$ \cite{Glob-math-annalen}, thus reducing the problem to a construction in $\B_N$ \cite{Glob-annals}. However Globevnik's construction in \cite{Glob-annals,Glob-math-annalen} did not provide with good topological properties of the connected components of the labyrinth, such as convexity or holomorphic contractibility. In \cite{AlaGlobLop} the authors simplified Globevnik's construction building a labyrinth in $\B_N$ whose connected components are balls in suitably chosen affine real hyperplanes. Alarc\'on and Forstneri\v{c} used a slight modification of this construction to obtain \cite[Theorem 1]{AlaFor}.

The main aim of this short note is to overcome the difficulty pointed out in \cite[Theorem 1]{AlaFor} and extend the construction made in \cite{AlaGlobLop} to pseudoconvex domains.


\begin{thm}\label{main-thm}
	Let $D$ be a pseudoconvex domain in $\C^N$ and let $(D_n)$ be an exhaustion of strictly pseudoconvex domains that are $\mathcal O(D)$-convex. Let also $(M_n)$ be a sequence of positive numbers. Then there are holomorphically contractible compact sets $\Gamma_n\subset D_{n+1}\setminus\overline D_{n}$ such that $\overline D_n\cup\bigcup_{j=n}^m \Gamma_j$ is $\mathcal O(D)$-convex for every $m\geq n$, and any continuous path connecting $\partial D_n$ to $\partial D_{n+1}$ and omitting $\Gamma_n$ has length greater than $M_n$.
	
	
	Moreover, if $D$ is Runge then it can be additionally assumed that each connected component of $\Gamma$ is the image of a $2N-1$-convex body under an automorphism of $\C^N$.
\end{thm}

It will follow from the proof that the $2N-1$-convex bodies appearing in the theorem above are $\mathbb R$-linear transformations of $2N-1$-dimensional balls. Proceeding as in \cite{AlaFor} one can use Theorem \ref{main-thm} to obtain the analogue of \cite[Theorem 1]{AlaFor} for Runge pseudoconvex domain, and thus answer the question posed by the authors.

The proof of Theorem \ref{main-thm} consists in three steps: first, to slightly adapt the main argument of \cite{AlaGlobLop} in order to build a suitable labyrinth in a strictly convex domain; second, to make use of a result by Diederich, Fornaess and Wold \cite{Died-For-Wold} to perform the construction in any strictly pseudoconvex domain, and three, to exhaust the given pseudoconvex domain by strictly pseudoconvex ones.


\section{Proof of Theorem \ref{main-thm}}

Let $D$ be a pseudoconvex domain of $\C^N$ and let us denote by $\mathcal{O}(D)$ the space of holomorphic functions on $D$.

Let us first recall some classical notions and results about polynomial and holomorphic convexity. We refer the reader to Stout's book \cite{Stout}.  Let $K$ and $L$ be compact subsets of a pseudoconvex domain $D$ of $\C^N$. We shall say that $K$ and $L$ are polynomially separated (respectively holomorphically separated with respect to $D$ - or simply $\mathcal{O}(D)$-separated) if there exists a holomorphic polynomial $P$ (resp. a function $f\in\mathcal{O}(D)$) such that $\widehat{p(K)}\cap \widehat{p(L)}$ is empty (resp. $\widehat{f(K)}\cap \widehat{f(L)}=\emptyset$). In particular if there exists a real hyperplane $\mathcal{H}$ such that $K$ is contained in one of the connected component of $\C^N\setminus \mathcal{H}$ and $L$ is contained in the other one, then $K$ and $L$ are polynomially separated. In this case, we will simply say that $K$ and $L$ are separated by a hyperplane. Note that if $K$ and $L$ are $\mathcal{O}(D)$-separated with $D=\C^N$, then $K$ and $L$ are polynomially separated.

It follows from a classical result of Kallin \cite[Theorem 1.6.19]{Stout}, that if $K$ and $L$ are both polynomially convex (resp. $\mathcal{O}(D)$-convex) and polynomially separated (resp. $\mathcal{O}(D)$-separated), then $K \cup L$ is polynomially convex (resp. $\mathcal{O}(D)$-convex).


\medskip{}

Let us now proceed with the proof of Theorem \ref{main-thm}. As a first step, we shall state an analogue of \cite[Theorem 1.5]{AlaGlobLop} for strictly convex domains.

\begin{lemme}\label{lem1}
	Let $D$ be a strictly convex domain in $\C^n$ and let $x\in \partial D$. There exists a neighborhood $U$ of $x$ such that for any $M>0$ and any compact set $K \subset D$ intersecting $U$, there exists a compact set $\Gamma\subset D$ with the following two properties:

	(i) $\Gamma$ can be written as a finite union $\bigcup_i \Gamma_i$, where each $\Gamma_i$ is a convex body in a hyperplane, and any $\Gamma_i$ is separated from $\bigcup_{j<i} \Gamma_j$ by a hyperplane.
	
	(ii) The length of any continuous path $\gamma:[0,1)\rightarrow D\setminus \Gamma$, with $\lim_{r\rightarrow}\gamma(r)\in\partial U \cap D$ and such that $\gamma(r_0)\in K\cap U$ for some $r_0\in[0,1)$ and  $\gamma(r)\in U$ for any $r_0<r<1$, is greater than $M$.
\end{lemme}

We only sketch the proof, as it is a simple modification of that of \cite[Theorem 1.5]{AlaGlobLop}. It makes use of \cite[Lemma 2.1]{AlaGlobLop}, that we recall below for notational convenience.

\begin{lemme}\label{lem2}There exist numbers $m\in \N$, $m\geq 2$, and $c\in \R$, $0<c<1/2$, depending only on $N$, such that for any $r>0$, there exist finitely many subsets $F_1,\ldots,F_m$ of $\S_N$ which satisfy the following:
	
	(i) $|p-q|\geq r$ for all $p,q\in F_j$, $p\neq q$, $j=1,\ldots,m$;
	
	(ii) If $F:=\bigcup_{j=1}^mF_j$ then $F\neq \emptyset$ and $\text{dist}(p,F)\leq cr$ for all $p\in \S_N$.
\end{lemme}

\begin{proof}[Outline of the proof of Lemma \ref{lem1}]Up to a translation and an $\mathbb R$-linear change of coordinates we can assume that $x=(1,0)\in\C\times \C^{n-1}$ and that near $x$ a defining function $r$ of $\partial D$ is of the form $r(z) = \Vert z\Vert ^2 -1 + o(\Vert z-1\Vert^2)$. Then there are two open balls $U_1$ and $U_2$ of $(1,0)$ and a diffeomorphism $\Phi:U_1\to U_2$ that maps $U_1 \cap D$ onto $U_2 \cap \B_N$. Upon shrinking the $U_i$'s, we can assume that $\Phi$ is arbitrarily close to the identity map in $\mathcal C^2$-topology. Let $U$ be any relatively compact ball in $U_1$, $(1,0)\in U$. 
	
Let $m$ and $c$ be given by Lemma \ref{lem2}. Following \cite{AlaGlobLop}, we fix a sequence $(s_j)$ of positive numbers, increasing, tending to $1$ and such that $\sum _j\sqrt{s_j - s_{j-1}} = \infty$. We set $s_{j,k} := s_{j-1} + k (s_j - s_{j-1})/(m+1)$.
Let $S_{j,k}$ denote $\Phi^{-1}(U_2 \cap s_{j,k} \mathbb S_N)$. Then there is a uniform constant $a>0$ such that any tangent ball 
with origin in $S_{j,k}$ and radius $r_j:=a \sqrt{s_j - s_{j-1}}$ does not intersect $S_{j, k+1}$, and any ball centred in $S_{j,k}\setminus U_1$ does not intersect $U$. We fix $t>1$ such that $tc<1/2$. Let now $F_1,\ldots,F_m$ be given by Lemma \ref{lem2} applied to $r:=2tr_j$ and denote by 
$E_{j,k}$ the set $\Phi^{-1} (s_{j,k}F_k \cap U_2)$. Let us denote by $\Gamma_{j,k,p}$ the tangent ball with origin $p\in E_{j,k}$ ans radius $r_j$. Observe that upon choosing $\Phi$ close enough to the identity map, in a way which depends only on $t>1$ - hence only on $N$, the sets $\Gamma_{j,k,p}$ can be separated by hyperplanes from $\bigcup_{(j',k',p')\prec(j,k,p)}\Gamma_{j',k',p'}$, where $\prec$ is the lexicographical order. 
Following the proof from \cite{AlaGlobLop} one easily checks that $\Gamma: =\bigcup_{j=1}^{J}\bigcup_{k,p} \Gamma_{j,k,p}$ satisfies the desired properties for some $J$ large enough.
\end{proof}

\begin{rem}\label{rem1}(1) Observe that each connected component of $\Gamma$ is the image of a $2N-1$ dimensional ball under an $\mathbb R$-linear isomorphism.
	
\noindent{}(2) Note that setting $\Gamma:=\bigcup_{j=J}^{J'}\bigcup_{k,p} \Gamma_{j,k,p}$ in the proof of Lemma \ref{lem2}, $J$ and $J'$ can be chosen large enough so that $\Gamma$ is in any given $\epsilon$-neighbourhood of $\partial D$ and separated by a hyperplane from any given compact set in $D$.
\end{rem}


The second step consists in extending \cite[Theorem 1.5]{AlaFor} to strictly pseudoconvex domains, using Lemma \ref{lem1}. This is the purpose of the next lemma.

\begin{lemme}\label{lem3}Let $D$ be a strictly pseudoconvex domain in $\C^N$ and let $K$ be a compact subset of $D$. Then for any $M>0$ there is a compact set $\Gamma$ in $D\setminus K$ such that $\Gamma \cup K$ is $\mathcal{O}(D)$-convex, with the property that any continuous path $\gamma:[0,1)\rightarrow D\setminus \Gamma$, $\gamma(0)\in K$ and $\lim_{r\rightarrow 1}\gamma(r) \in \partial D$, has length greater than $M$. 
%

$\Gamma$ can be chosen so that each of its connected components is holomorphically contractible. If additionally $D$ is Runge, the connected components of $\Gamma$ can even be chosen as the images of $2N-1$-dimensional balls under $\R$-linear isomorphisms and automorphisms of $\mathbb C^N$ .
\end{lemme}

\begin{proof}
Let $K$ be fixed as in the statement. By \cite[Theorem 1.1]{Died-For-Wold}, 
there exists an open ball $U$ centred at $x$ 
and a holomorphic embedding $\Phi_x:\overline D \to \overline \B_N$ such that $\Phi_x(U)$ and some $\Gamma' \subset \Phi_x(D)$ satisfy the conclusion of Lemma \ref{lem1} for some constant $M'>0$.
Upon choosing $M'$ large enough - in a way which depends only on $\Phi_x$, the set $\Gamma := \Phi_x^{-1}(\Gamma')$ satisfies that any continuous path $\gamma$ in $D\setminus \Gamma$ connecting $K$ to $U\cap \partial D$ and satisfying $\gamma(r_0)\in K'\cap U$ for some $r_0\in [0,1)$ and $\gamma(r)\in U$, $r_0<r<1$, has length greater than $M$. Moreover, by Remark \ref{rem1} (2) and Kallin's theorem, $\Gamma$ can also be chosen so that $\Gamma \cup K$ is $\mathcal{O}(D)$-convex and contained in an $\epsilon$-neighbourhood of $\partial D$ for any given $\epsilon >0$.

Let us then consider a finite covering of $\partial D$ by such open balls $U_{x_1},\ldots,U_{x_k}$. Upon slightly shrinking the $U_{x_j}$'s, we may and shall assume that there exist $\delta, \eta >0$ and an open $\eta$-neighbourhood $V$ of $\partial D$ such that the distance between $ V \cap \partial U_{x_j}$ and $ V\cap \partial (\bigcup_{i\neq j} U_{x_i})$ is greater than $\delta$ for $j=1,\ldots, k$. From now on, let us fix $K'=D\setminus V$. Observe that $K'\cap U_j\neq \emptyset$ for any $j$ and that upon choosing $\eta$ small enough, we shall assume that $K\subset K'$. Let us enumerate the sets $U_{x_i}$ as a sequence $(U_j)$ in such a way that for any $i=1,\ldots,k$ there exist infinitely many $j$ such that $U_j=U_{x_i}$. We denote by $\Phi_j$ the mapping corresponding to $U_j$ and given by \cite{Died-For-Wold}. Following the above procedure with $K'=D\setminus V$, we build a labyrinth $\Gamma_1$ in $D$ and such that:

\begin{enumerate}\item [(i)] {$\Gamma_1$ is contained in $V$} and $\Gamma _1 \cup K$ is $\mathcal{O}(D)$-convex;
\item [(ii)] Any continuous path $\gamma$ in $D\setminus \Gamma _1$ connecting $K$ to $U_1\cap \partial D$ and satisfying $\gamma(r_0)\in K'\cap U_1$ for some $r_0\in [0,1)$ and $\gamma(r)\in U_1$, $r_0<r<1$,  has length greater than $M$.
\end{enumerate}

Assuming that $\Gamma_1,\ldots,\Gamma_j$ have been built, we build $\Gamma_{j+1}$ in $D$ such that:

\begin{enumerate}\item [(i)] $\Gamma_{j+1}$ is contained in an $\eta_j$-neighborhood of $\partial D$, where $\eta_j$ is the distance from $\partial D$ to $\Gamma_1\cup\ldots \cup \Gamma_j$, and $\bigcup_{i=1}^{j+1}\Gamma _{i} \cup K$ is $\mathcal{O}(D)$-convex;
\item [(ii)] Any continuous path $\gamma$ in $D\setminus \Gamma _{j+1}$ connecting $K$ to $U_{j+1}\cap \partial D$ and satisfying $\gamma(r_0)\in K'\cap U_{j+1}$ for some $r_0\in [0,1)$ and $\gamma(r)\in U_{j+1}$, $r_0<r<1$,  has length greater than $M$.
\end{enumerate}

It is now easily checked that there exists $J\in \N$ big enough so that $\Gamma := \bigcup_{j=1}^J \Gamma_j$ has the desired property. Actually, it is enough to take $J$ such that each $U_{x_j}$ appears in a sequence $(U_j)$ at least $n$ times, where $n\delta >M$. Indeed, let $\gamma$ be a path in $D$ with $\gamma(0)\in K$ and $\lim _{r\rightarrow 1}\gamma(r)\in \partial D$. Since $K\subset D\setminus V$ and $\gamma$ is continuous, without loss of generality, we may assume, up to re-parametrization, that $\gamma([0,1))\subset V$. 
If there exists $0<r_1<r_2<1$ and $j\leq J$ such that $\gamma(r)\in U_j$ for any $r_1\leq r\leq r_2$ and $\gamma(r_1)\in V_{j-1}$ and $\gamma (r_2) \in V_j$, where $V_j$ is the $\eta_j$-neighbourhood appearing in the construction, then the length of $\gamma$ is clearly greater than $M$. 
If not, it means that $\gamma$ has to escape from some $U_j$ as many times as it may have to pass through some $\Gamma_j$. With $J$ chosen as above, $\gamma$ would then have to pass at least $n$ times from a $U_j$ to another. Since the image of $\gamma$ is in $V$ and the distance between $ V \cap \partial U_{x_j}$ and $ V\cap \partial (\bigcup_{i\neq j} U_{x_i})$ is greater than $\delta$ for $j=1,\ldots, k$, the length of $\gamma$ has to bigger than $n\delta$.

The $\mathcal{O}(D)$-convexity of $\Gamma \cup K$ proceeds from the construction and Kallin's theorem recalled above. Observe that the last assertion of the lemma directly follows from the construction.
%
%
%

\end{proof}

The third and last step is straightforward: Given $D$ a pseudoconvex domain, we consider an exhaustion $(D_n)$ of $D$ by $\mathcal{O}(D)$-convex strictly pseudoconvex domains (Runge strictly pseudoconvex domains if $D$ is Runge) and inductively apply Lemma \ref{lem3}. For the existence of such an exhaustion, we refer to \cite{Forstneric}.



\end{document}